\documentclass[12pt]{amsart}
\usepackage{amscd,amsmath,amsthm,amssymb}
\usepackage{pstcol,pst-plot,pst-3d}
\usepackage{color}
\usepackage{pstricks}
\usepackage{stmaryrd}
\newpsstyle{fatline}{linewidth=1.5pt}
\newpsstyle{fyp}{fillstyle=solid,fillcolor=verylight}
\definecolor{verylight}{gray}{0.97}
\definecolor{light}{gray}{0.9}
\definecolor{medium}{gray}{0.85}
\definecolor{dark}{gray}{0.6}

 \def\NZQ{\mathbb}               

 \def\ZZ{{\NZQ Z}}
 \def\RR{{\NZQ R}}

 \def\frk{\mathfrak}               

 \def\mm{{\frk m}}

 \def\G{{\mathcal G}}

 \def\D{{\mathcal D}}

 \def\Sc{{\mathcal S}}
 \def\Tc{{\mathcal T}}

 \def\bb{{\mathbf b}}

 \def\eb{{\mathbf e}}
 \def\opn#1#2{\def#1{\operatorname{#2}}} 
 \opn\chara{char} \opn\length{\ell} \opn\pd{pd} \opn\rk{rk}
 \opn\projdim{proj\,dim} \opn\injdim{inj\,dim} \opn\rank{rank}
 \opn\depth{depth} \opn\grade{grade} \opn\height{height}
 \opn\embdim{emb\,dim} \opn\codim{codim}
 
 \opn\Tr{Tr} \opn\bigrank{big\,rank}
 \opn\superheight{superheight}\opn\lcm{lcm}
 \opn\trdeg{tr\,deg}
 \opn\reg{reg} \opn\lreg{lreg} \opn\ini{in} \opn\lpd{lpd}
 \opn\size{size} \opn\sdepth{sdepth}
 \opn\link{link}\opn\fdepth{fdepth}\opn\lex{lex}
 \opn\tr{tr}
 \opn\type{type}
 \opn\gap{gap}
 \opn\depth{depth}

 \opn\div{div} \opn\Div{Div} \opn\cl{cl} \opn\Cl{Cl}
 \opn\Spec{Spec} \opn\Supp{Supp} \opn\supp{supp} \opn\Sing{Sing}
 \opn\Ass{Ass} \opn\Min{Min}\opn\Mon{Mon}
 \opn\Im{Im} \opn\Ker{Ker} \opn\Coker{Coker} \opn\Am{Am}
 \opn\Hom{Hom} \opn\Tor{Tor} \opn\Ext{Ext} \opn\End{End}
 \opn\Aut{Aut} \opn\id{id}
 
 \opn\nat{nat}
 \opn\pff{pf}
 \opn\Pf{Pf} \opn\GL{GL} \opn\SL{SL} \opn\mod{mod} \opn\ord{ord}
 \opn\Gin{Gin} \opn\Hilb{Hilb}\opn\sort{sort}
 \opn\PF{PF}\opn\Ap{Ap}
 \opn\aff{aff}
 \opn\relint{relint} \opn\st{st}
 \opn\lk{lk} \opn\cn{cn} \opn\core{core} \opn\vol{vol}  \opn\inp{inp} \opn\nilpot{nilpot}
 \opn\link{link} \opn\star{star}\opn\lex{lex}\opn\set{set}
 \opn\width{wd}
 \opn\Fr{F}
 \opn\QF{QF}
 \opn\G{G}
 \opn\type{type}\opn\res{res}
 \opn\conv{conv}
 \opn\gr{gr}
 
 \def\pot#1#2{#1[\kern-0.28ex[#2]\kern-0.28ex]}

 \opn\dirlim{\underrightarrow{\lim}}
 \opn\inivlim{\underleftarrow{\lim}}
 \let\sect=\cap
 \let\dirsum=\oplus
 
 \let\iso=\cong

 \let\Dirsum=\bigoplus
 
 \let\to=\rightarrow
 
 \def\Implies{\ifmmode\Longrightarrow \else
         \unskip${}\Longrightarrow{}$\ignorespaces\fi}
 \def\implies{\ifmmode\Rightarrow \else
         \unskip${}\Rightarrow{}$\ignorespaces\fi}
 \def\iff{\ifmmode\Longleftrightarrow \else
         \unskip${}\Longleftrightarrow{}$\ignorespaces\fi}

 \let\:=\colon
 \newtheorem{Theorem}{Theorem}[section]
 \newtheorem{Lemma}[Theorem]{Lemma}
 \newtheorem{Corollary}[Theorem]{Corollary}

 \let\epsilon\varepsilon
 \let\kappa=\varkappa
 \def\qed{\ifhmode\textqed\fi
       \ifmmode\ifinner\quad\qedsymbol\else\dispqed\fi\fi}
 \def\textqed{\unskip\nobreak\penalty50
        \hskip2em\hbox{}\nobreak\hfil\qedsymbol
        \parfillskip=0pt \finalhyphendemerits=0}
 \def\dispqed{\rlap{\qquad\qedsymbol}}
 
 \opn\dis{dis}
 \def\pnt{{\raise0.5mm\hbox{\large\bf.}}}
 
 \opn\Lex{Lex}

\begin{document}
\title {On the fiber cone of monomial ideals}

\author {J\"urgen Herzog and  Guangjun Zhu$^{^*}$}

\address{J\"urgen Herzog, Fachbereich Mathematik, Universit\"at Duisburg-Essen, Campus Essen, 45117
Essen, Germany} \email{juergen.herzog@uni-essen.de}

\address{Guangjun Zhu, School of Mathematical Sciences, Soochow University,
 Suzhou 215006, P. R. China}\email{zhuguangjun@suda.edu.cn(Corresponding author:Guangjun Zhu)}

\dedicatory{ }

\begin{abstract}
We consider the fiber cone of monomial ideals. It is shown that for monomial ideals $I\subset K[x,y]$ of height $2$,  generated by $3$ elements,  the fiber cone $F(I)$ of $I$ is a hypersurface ring, and that $F(I)$  has positive depth for interesting  classes of height $2$  monomial ideals $I\subset K[x,y]$,  which are   generated by $4$ elements.  For these classes of ideals we also  show that $F(I)$ is Cohen--Macaulay if and only if the defining ideal $J$ of $F(I)$ is generated by at most 3 elements. In all the cases a minimal set of generators of $J$ is determined.
\end{abstract}

\thanks{The paper was written while the second author was visiting the Department of Mathematics of University
Duisburg-Essen. She  spent a memorable time at Essen, so she would like to express
her hearty thanks to Maja for hospitality. She also wishes to thank for the hospitality of
Department of Mathematics of University Duisburg-Essen, Germany.\\
\hspace{5mm}* Corresponding author.}




\maketitle

\setcounter{tocdepth}{1}

\section*{Introduction}
Let $I\subset S$ be a graded ideal in the polynomial ring $S=K[x_1,\ldots,x_n]$ over the field $K$. Let $\mm=(x_1,\ldots,x_n)$ be the graded maximal ideal of $S$, and $R(I)=\Dirsum_{j\geq 0}I^j$ the Rees algebra of $I$. Then the fiber cone of $I$ is the standard graded $K$-algebra $F(I)=R(I)/\mm R(I)$. Let $\mu(L)$ denote the minimal number of generators of a  graded ideal $L\subset S$.  Motivated by the fact that $\mu(I^k)\leq \mu(I^{k+1})$ for all $k\geq 1$, if $\depth F(I)>0$, we ask  when  $F(I)$ have positive depth. In concrete cases this question is hard to  answer.

In this paper we focus on monomial ideals $I$ of height $2$ in a polynomial with two variables. It turns out that even this case, which is the simplest possible to consider,  the problem is pretty hard. For example, $\depth F(I)=0$  for  the ideal $I=(x^{25}, x^{20}y^{5}, x^{19}y^{19}, x^{5}y^{20}, y^{25})\subset K[x,y]$.  On the other hand, by checking hundreds of examples, we could not find any monomial ideal $I\subset K[x,y]$ of height $2$ with $\mu(I)\leq 4$ and $\depth F(I)=0$. Thus we expect that $\depth F(I)>0$, if $\mu(I)\leq 4$. This is trivially true when $\mu(I)=2$, in which case $F(I)$ is a polynomial ring. For the case $\mu(I)=3$, we show that  $F(I)$ is a hypersurface ring, and hence is Cohen--Macaulay.  This is the content of Section $2$.

For a Noetherian local ring with infinite residue class field, Heinzer and Kim \cite[Proposition 5.4]{HK} give several equivalent conditions for $F(I)$  being a hypersurface ring, and in \cite[Theorem 5.6 ]{HK}  they provide a sufficient condition for this  in terms of a depth condition for the associated graded ring of $I$. But all these criteria are hard to apply in our concrete case. Our approach in this,  and also in the other cases, considered later in this paper, is more straightforward, and is aimed at computing the generators of the defining ideal $J$ of $F(I)$ explicitly.  This approach always leads to the problem to find integer solutions to linear inequalities with integer coefficients. Of course, in general,  this is a difficult problem, and hence we can present only partial results for the depth problem of the fiber cone  of a monomial ideal $I\subset K[x,y]$ for height $2$ and with $4$ generators.

In Section $3$ we consider symmetric ideals, that is ideals $I\subset K[x,y]$, $I=(x^c, x^by^a, x^ay^b, y^c)$ with $c>b>a>0$ and $\gcd(a,b,c)=1$.  In \cite{HMQ} it is shown that in this case $F(I)$ is Cohen--Macaulay if and only if $b=a+1$, and this is the case if and only if $\mu(J)\leq 3$. More generally,  the Cohen--Macaulayness of $F(I)$ holds for  $F(I)\cong K[x^c, x^{c-a_1}y^{a_1}, x^{c-a_2}y^{a_2}, y^c]$ with $0<a_1<a_2< c$, if and only if $\mu(J)\leq 3$,  as was shown by    Bresinsky,  Schenzel and Vogel \cite[Lemma  4]{BSV}. Here, as before, $J$ denotes  the defining ideal  of $F(I)$. Note that such rings may be viewed as  the coordinate ring of a projective monomial curve. Projective monomial curves have been studied in many papers, see for example \cite{BR}, \cite{CN}, \cite{HS}, \cite{MPT} and  \cite{RR}.  For the symmetric ideal $I$, the fiber cone $F(I)$ is the coordinate ring of a projective monomial curve only if $a+b=c$.  Yet, also when $a+b\neq c$,  it is shown in Corollary~\ref{noteasy} that $F(I)$ is Cohen--Macaulay, if and only if $\mu(J)\leq 3$.

In the last section we study the fiber cone of monomial ideals of the form $I=(x^{2a},x^{a}y^{b},x^{c}y^{d},y^{2b})$. It turns out that the fiber cone of this type of ideals is a complete intersection. The different cases to be discussed are treated in Theorem~\ref{Romania}, Theorem~\ref{kolon} and Theorem~\ref{Hibinotcome}, respectively.

We expect that the results regarding the  depth of the fiber cone  obtained for the monomial ideals considered in this paper   hold for all monomial ideals of the form $I=(x^a,x^cy^d,x^ey^f,y^b)$. In other words, we expect that for any ideal of this type we always have  $\depth F(I)>0$,  and that $F(I)$ is Cohen-Macaulay, if and only if the defining ideal $J$ of $F(I)$ is generated by at most $3$ elements.

\medskip
\section{Generalities about the relations of the fiber cone}
Let $K$ be a field, $S=K[x_1,\ldots,x_n]$ the polynomial ring in $n$ indeterminates with graded maximal ideal $\mm=(x_1,\ldots,x_n)$, and  $I\subset S$ a
monomial ideal with minimal set of monomial generators $G(I)=\{u_1,\ldots,u_m\}$.  We denote by $R(I)=\Dirsum_{j\geq 0}I^j$ the Rees algebra of $I$ and by
$F(I)=\Dirsum_{j\geq 0}I^j/\mm I^j$ the fiber cone of $I$.

Let $F(I)\cong T/J$, where  $T=K[z_1, \ldots, z_{m}]$ is the polynomial ring over $K$, and $J$ is  the kernel of the $K$-algebra  homomorphism $T\to F(I)$
defined by $z_i\mapsto u_i+\mm I$ for  $i=1,\ldots,m$. The relations of $F(I)$ can be obtained from the relations of the Rees ring by reduction modulo
$\mm$. The Rees ring $R(I)$ of $I$ is a toric  standard graded $S$-algebra. Therefore, the
 relations belonging to a minimal set of generators of  the defining ideal  $L$  of   $R(I)$  are  of the form
\begin{eqnarray}
\label{relation}
f=u\prod_{i=1}^{m}z_i^{r_i}-v\prod_{i=1}^{m}z_i^{s_i},
\end{eqnarray}
with $\gcd(u\prod_{i=1}^{m}z_i^{r_i}, v\prod_{i=1}^{m}z_i^{s_i})=1$, where  $u$  and $v$  are monomials in $S$, satisfying
\[
u\prod_{i=1}^{m} u_i^{r_i}=v\prod_{i=1}^{m}u_i^{s_i}.
\]
In addition one has
\[
 \sum_{i=1}^{m }r_i=\sum_{i=1}^{m}s_i>0.
\]
In particular it follows that $J$ is generated by  monomials and homogeneous binomials.

\medskip
\section{The case when $F(I)$ is a hypersurface}
We turn to the  special case described in the following
\begin{Lemma}
\label{special}
Let $I\subset S$  be the monomial ideal with $G(I)=\{u_1,\ldots,u_{n+1}\}$, where $u_i=x_i^{a_i}$ for $i=1,\dots, n$ and $u_{n+1}=x_1^{b_1}\cdots
x_n^{b_n}$ with $0<b_i<a_i$ for $i=1,\ldots,n$. If $\prod\limits_{i=1}^{n+1}z_i^{r_i}\in J$ with $r_{n+1}>0$, then $z_n^{r_{n+1}}\in J$.
\end{Lemma}
\begin{proof}
In view of (\ref{relation}) there exists $f\in L$ of the form $f=\prod\limits_{i=1}^{n+1}z_i^{r_i}-v\prod_{i=1}^{n}z_i^{s_i}$.
Suppose that $r_i>0$ for some $i<n+1$. Since $\gcd(\prod\limits_{i=1}^{n+1}z_i^{r_i}, v\prod_{i=1}^{n}z_i^{s_i})=1$, we may further assume that
$f=(\prod_{i=1}^{t}z_i^{r_i})z_{n+1}^{r_{n+1}}-v\prod_{i=t+1}^{n}z_i^{s_i}$ with $r_{n+1}=\sum_{i=t+1}^{n}s_i-c$, where $c=\sum_{i=1}^{t}r_i>0$.
It  follows that
\[
\prod_{i=t+1}^{n}x_i^{b_ir_{n+1}}=v'\prod_{i=t+1}^{n}x_i^{a_is_i},
\]
 where  $v'$  is a monomial in $S$,
which implies that $b_ir_{n+1}\geq a_is_i$ for  $i=t+1,\ldots,n$.
We consider the following three cases:

(i) There exists some $j\in \{t+1,\ldots,n\}$ such that $s_j\geq c$.
Then  $f=z_{n+1}^{r_{n+1}}-wz_j^{s_j-c}\prod\limits_{\begin{subarray}{c}
i=t+1\\
i\neq j
\end{subarray}}^{n}z_i^{s_i}$ with
 $w=v'x_j^{a_jc}\prod_{i=1}^{t}x_i^{b_ir_{n+1}}$  belongs to  $L$.

(ii)  $s_j<c$ for all $j\in \{t+1,\ldots,n\}$  and there exist $i_1,\ldots,i_m\in \{t+1,\ldots,n\}$ such that   $c=\sum\limits_{j=1}^{m}s_{i_j}$.
Then   $f=z_{n+1}^{r_{n+1}}-w\prod\limits_{\begin{subarray}{c}
i=t+1\\
i\notin\{i_1,\ldots,i_m\}
\end{subarray}}^{n}z_i^{s_i}$ with  $w=v'(\prod\limits_{j=1}^{m}x_{ij}^{a_{ij}s_{ij}})(\prod\limits_{i=1}^{t}x_i^{b_ir_{n+1}})$ belongs to   $L$.

(iii) There exists some $m$ such that $s_{1}+\dots+s_{m}<c<s_{1}+\dots+s_{{m+1}}$.
Then   $f=z_{n+1}^{r_{n+1}}-w(\prod\limits_{
i=m+2}^{n}z_i^{s_i})z_{m+1}^{s_{m+1}-(c-\sum\limits_{i=1}^{m}s_{i})}$ with
$w=v'(\prod\limits_{j=1}^{m}x_j^{a_js_j})(\prod\limits_{i=1}^{t}x_i^{b_ir_{n+1}})x_{m+1}^{a_{m+1}(c-\sum\limits_{i=1}^{m}s_{i})}$ belongs to  $L$.

In all three cases it follows that $z_{n+1}^{r_{n+1}}\in J$, as desired.
\end{proof}

 Let $H$ be the hyperplane in $\RR^n$  passing through the points $a_i\eb_i$, $i=1,\ldots,n$,  where $\eb_1,\ldots,\eb_n$ is the standard unit basis of
 $\RR^n$. Then
\[
H=\{(b_1,\ldots,b_n)\in \RR^n\: \sum_{i=1}^n\frac{b_i}{a_i}=1\}.
\]
Let   $H^{\pm}$ be  the two open half spaces defined by $H$.  Then
\[
H^+=\{(b_1,\ldots,b_n)\in \RR^n\: \sum_{i=1}^n\frac{b_i}{a_i}>1\} \quad  \text{and} \quad H^-=\{(b_1,\ldots,b_n)\in \RR^n\:
\sum_{i=1}^n\frac{b_i}{a_i}<1\}.
\]

 With this  notation introduced we now have
\begin{Theorem}\label{$n+1$generators}
Let  $I=(x_1^{a_1},\ldots,x_n^{a_n},\prod\limits_{i=1}^{n}x_i^{b_i})\subset  S=K[x_1,\ldots,x_n]$ with $a_i>b_i>0$ for $i=1,\dots,n$, and let
$\bb=(b_1,\ldots,b_n)$.
Then there exist integers $r_i$ such that
\[
J= \left\{
\begin{array}{ll}
(z_{n+1}^{r_{n+1}}-\prod_{i=1}^nz_i^{r_i} ), &   \text{if}\ \ \bb\in H,\\
(z_{n+1}^{r_{n+1}}), & \text{if} \ \ \bb\in H^+,\\
(\prod_{i=1}^nz_i^{r_i}), & \text{if}\ \ \bb\in H^-.
\end{array}
\right.
 \]
\end{Theorem}

\begin{proof} Let $f=u\prod_{i=1}^{n+1}z_i^{r_i}-v\prod_{i=1}^{n+1}z_i^{s_i}$ be a binomial as in (\ref{relation}) that is a minimal generator of $L$.
Then $\gcd(\prod_{i=1}^{n+1}z_i^{r_i},  \prod_{i=1}^{n+1}z_i^{s_i})=1$. Hence we may assume  that
$f=u\prod_{i=t+1}^{n+1}z_i^{r_i}-v\prod_{i=1}^{t}z_i^{s_i}$.

Notice that  $f$  induces a non-zero  element in  $J$ if and only if $u=1$ or $v=1$, and all generators of $J$ are obtained in this way.

We claim: $r_{n+1}>0$.  Indeed, assume first that  $u=1$, then $f=\prod_{i=t+1}^{n+1}z_i^{r_i}-v\prod_{i=1}^{t}z_i^{s_i}$. If  $r_{n+1}=0$, then
$\prod_{i=t+1}^{n}x_i^{a_ir_i}=v\prod_{i=1}^{t}x_i^{a_is_i}$.  It follows that $v=(\prod_{i=t+1}^{n}x_i^{a_ir_i})v'$ for some monomial $v'$. This implies
that $1=v' \prod_{i=1}^tx_i^{a_is_i}$, so that $s_i=0$ for all $i$, a contradiction.   Similar argument for $v=1$.

Assume first that  $\bb\in H$. In this case  $F(I)\iso K[x_1^{a_1},\ldots,x_n^{a_n},\prod_{i=1}^nx_i^{b_i}]$, and hence $F(I)$ is a domain and $J$ is a
prime ideal of height $1$. Therefore,  $J$  does not contain any monomial generator, and so  $f= \prod_{i=t+1}^{n+1}z_i^{r_i}- \prod_{i=1}^tz_i^{s_i}$.
It follows that
$$
\prod_{i=1}^tx_i^{b_ir_{n+1}}\prod_{i=t+1}^{n}x_i^{a_ir_i+b_ir_{n+1}}=\prod_{i=1}^tx_i^{a_is_i}.
$$
This implies that $t=n$, so that $f= z_{n+1}^{r_{n+1}}-\prod_{i=1}^{n}z_i^{s_i}$.

Since $f$ is a minimal generator of $J$ and since $J$ is a prime ideal, it follows that $f$ is irreducible. Therefore,  the principal ideal $(f)\subset J$
is a prime ideal. Since $J$ is a prime ideal of height $1$,
this implies that  $J=(f)$.

Next we consider the case that  $\bb \in H^+$. In this case  we see that $u=1$ and $v\neq 1$. Then it follows that the generators of $J$ are induced from
binomials of the form $f=\prod_{i=t+1}^{n+1}z_i^{r_i}-v\prod_{i=1}^{t}z_i^{s_i}$  with  $r_{n+1}>0$.
This implies that  $J$ is generated by monomials of the form  $z_{n+1}^{r_{n+1}}\prod_{i=1}^nz_i^{r_i}$ with $r_{n+1}>0$.  By applying Lemma~\ref{special},
it follows that $J=(w_1,\ldots,w_m)$ with $w_j=z_{n+1}^{r_{n+1,j}}$ for $j=1,\ldots,m$. Let $r_{n+1}=\min\{r_{n+1,j}\:\; j=1,\ldots,m\}$. Then
$J=(z_{n+1}^{r_{n+1}})$.

Finally assume that $\bb\in H^-$. Then  $u\neq 1$ and $v=1$. In this case, all generators of $J$ are of the form $z_1^{t_1}\cdots z_n^{t_n}$. Say,
$J=(w_1,\ldots,w_m)$ with  $w_j=z_1^{t_{1,j}}\cdots z_n^{t_{n,j}}$ for $j=1,\ldots,m$. Let $w=z_1^{s_1}\cdots z_n^{s_n}$ with
$s_i=\min\{t_{i,1},\ldots,t_{i,m}\}$.  Then  $w_j=v_jw$ for $j=1,\ldots, m$ and monomials $v_j$ in $T$. We will show that $w\in J$. Then this implies
$J=(w)$.

Indeed, since $w_j\in J$,  there exists a monomial  $u_j\in S$ such that  $u_jz_{n+1}^{r_{n+1,j}}\prod_{i=t+1}^{n}z_i^{r_{i,j}}-w_j\in L$. Claim:
$r_{i,j}=0$  for all $i=t+1,\ldots,n$, $j=1,\ldots,m$.
Otherwise, there exist some  $i\in \{t+1,\ldots,n\}$ such that $r_{i,j}>0$. Without loss of generality, we may assume that $r_{i,j}>0$ for $i=t+1,\ldots,
n$. This implies that $u_j(\prod_{i=1}^{n}x_{i}^{b_ir_{n+1,j}})(\prod_{i=t+1}^{n}x_i^{a_ir_{i,j}})=\prod_{i=1}^{n}x_{i}^{a_{i}t_{i,j}}$. It follows that
$t_{i,j}>0$ for $i=t+1,\ldots, n$. Then $\gcd(u_jz_{n+1}^{r_{n+1,j}}\prod_{i=t+1}^{n}z_i^{r_{i,j}},\prod_{i=1}^{n}z_i^{t_{i,j}})\neq 1$, a contradiction.

Therefore, for any  $w_j\in J$,  there exists a monomial  $u_j\in S$ such that $u_jz_{n+1}^{r_{n+1,j}}-w_j\in L$,  where $r_{n+1,j}=\sum_{i=1}^nt_{i,j}$.
This is equivalent to saying that $b_i(\sum_{k=1}^nt_{k,j})\leq a_it_{i,j}$ for any  $i=1,\ldots,n$ and $j=1,\ldots,m$.  Now we have
\[
b_i(\sum_{k=1}^ns_k)\leq \min_{j=1,\ldots,m}\{ b_i(\sum_{k=1}^nt_{k,j})\}\le \min_{j=1,\ldots,m}\{a_it_{i,j}\}= a_is_i.
\]
This shows that there exists a monomial $u\in S$ such that  $uz_{n+1}^{r_{n+1}}-w\in L$, where $r_{n+1}=\sum_{k=1}^ns_k$. It remains to be shown that
$u\neq 1$. Then this implies that $w\in J$. Suppose  $u=1$. Then $b_ir_{n+1}=a_is_i$ for all $i$. This implies that $\sum_{i=1}^n\frac{b_i}{a_i}=1$.
Therefore, $\bb \in H$, a contradiction.
\end{proof}

\begin{Corollary}
\label{hyper}
With the assumptions and notation  of Theorem~\ref{$n+1$generators},  the fiber cone $F(I)$ is  a hypersurface ring. Moreover, $F(I)$ is a domain if and
only if $\bb\in H$. The fiber cone $F(I)$  has exactly one  non-zero minimal prime ideal if $\bb\in H^+$, and  it has precisely $n$ non-zero minimal prime
ideals if $\bb\in H^-$.
\end{Corollary}

The last sentence of the statement of Corollary \ref{hyper} also follows from \cite[Corollary 4.6]{S}. This corollary implies in particular, that $\depth F(I)=2$, if $I\subset K[x,y]$ is a  non principal monomial ideal  of $\height 2$ generated by 3 elements. There exist  some
examples of non principal monomial ideals $I\subset K[x,y]$  of $\height 2$ with more than 4 generators and $\depth F(I)=0$.

\medskip
\section{Symmetric ideals}
In the following, we study in more detail those symmetric ideals $I$ with $\mu(I)=4$.
We fix the following   notation.  Let $0<a<b<c$ be integers with $\gcd(a,b,c)=1$.
Then we define the symmetric ideal $I=(x^c,x^by^a,x^ay^b,y^c)$. We let $T= K[z_1,\ldots,z_4]$ be the polynomial ring,  and let   $J$ be the kernel of the
canonical map $T\to F(I)$ with  $z_i\mapsto u_i$ for $i=1, \ldots,4$,   where    $u_1=x^c$, $u_2=x^by^a$, $u_3=x^{a}y^{b}$ and $u_4=y^c$.

\begin{Theorem}\label{symmetricidealup}
Let $I=(x^c,x^by^a,x^ay^b,y^c)\subset K[x,y]$ be a symmetric monomial ideal with $b+a>c$. Then we have
\begin{enumerate}
\item[ (a)] The ideal $J$ is minimally generated as given in  one of the following four cases, and each of these cases occur.
\begin{enumerate}
\item[(i)] $J=(z_2z_3, z_2^r, z_3^r)$;
\item[(ii)] $J=(z_2z_3, z_2^r, z_3^r)+L$, where $L$ is minimally generated by monomials of the form $z_1^iz_3^j$ and $z_2^jz_4^i$ with $i,j>0$;
\item[(iii)] $J=(z_2z_3, z_2^{m+1}, z_3^{m+1}, z_1^\ell z_3^m-z_2^mz_4^\ell)$, where
  $\ell=\frac{b-a}{\gcd(c,b-a)}$ and\\ $m=\frac{c}{\gcd(c,b-a)}$.
\item[(iv)] $J=(z_2z_3, z_2^{r}, z_3^{r}, z_1^\ell z_3^m-z_2^mz_4^\ell)+L'$, where $L'$ is minimally generated by monomials of the form $z_1^{i'}z_3^{j'}$ and $z_2^{j'}z_4^{i'}$ with $i',j'>0$,
  $\ell=\frac{b-a}{\gcd(c,b-a)}$ and  $m=\frac{c}{\gcd(c,b-a)}$.
\end{enumerate}
\item[(b)] $F(I)$ is Cohen--Macaulay  if and only if $J$ is generated as in case {\em (i)}. Otherwise,  $\depth F(I)=1$.
\end{enumerate}
\end{Theorem}
\begin{proof}
(a) We first show that the monomials $z_2z_3, z_2^r, z_3^r$ belong to $J$ in all four   cases. Indeed, we have
 $u_2u_3=x^{a+b-c}y^{a+b-c}u_1u_4$, which implies that $z_2z_3\in J$.

Since   $b+a>c$, we have $\frac{a}{c-a}>\frac{c-b}{b}$. Thus there exist   positive integers $m$ and $n$ such that
\begin{eqnarray}
\label{essen}\frac{c-b}{b}m\leq n\leq \frac{a}{c-a}m.
\end{eqnarray}
This implies that  $br\geq cm$ and $ar\geq cn$, where $r=m+n$. Moreover one of the two inequalities must be strict, because in (\ref{essen}) one of the two
inequalities must be strict.  This implies that $z_2^{r}-vz_1^{m}z_4^{n}$ is a relation of $R(I)$ with a monomial $v\neq 1$. Hence $z_2^r \in J$. Of course
we may assume that $r$ is the smallest integer with $z_2^r\in J$. By symmetry we also have $z_3^r\in J$.

Case (i) can happen for example, when  $I=(x^4,x^3y^2,x^2y^3, y^4)$. In this example,  $J=(z_2z_3, z_2^2, z_3^2)$.

Next we show that if $J$ is generated by monomials, then $J$ must be of the form (i) or (ii). Indeed, let $u$ be any other monomial generator of $J$ which
is not of the form given in (ii).  Then $u=z_1^{i}z_2^{j}z_4^{k}$ with $i>0$ or $u=z_1^{i}z_3^{j}z_4^{k}$ with $k>0$. By symmetry it is enough to show that
$u=z_1^{i}z_2^{j}z_4^{k}$ with $i>0$ is not a minimal generator of $J$.

More generally  we show that elements of the following type  do not belong to any minimal set of generators of $J$:
(1) $z_1^{i}z_2^{j}z_3^{k}$, $z_2^{i}z_3^{j}z_4^{k}$ with $i,j,k>0$; (2) $z_1^{r}$, $z_4^{r}$, $z_1^{r}-z_2^{i}z_3^{j}z_4^{k}$,
$z_4^{r}-z_1^{i}z_2^{j}z_3^{k}$; (3) $z_1^{i}z_2^{j}$,   $z_3^{k}z_4^{\ell}$, $z_1^{i}z_2^{j}-z_3^{k}z_4^{\ell}$ with $i,\ell>0$;
(4) $z_1^{i}z_2^{j}z_4^{k}$ with $i>0$,    $z_1^{i}z_3^{j}z_4^{k}$ with $k>0$, $z_3^{r}-z_1^{i}z_2^{j}z_4^{k}$, $z_2^{r}-z_1^{i}z_3^{j}z_4^{k}$;
(5) $z_1^{i}z_4^{j}$,  $z_1^{i}z_4^{j}-z_2^{k}z_3^{\ell}$ with $i,j>0$.

Indeed, (1)  follows from the fact  $z_2z_3\in J$.

By symmetry,   type (2) reduces to $z_1^{r}$, $z_1^{r}-z_2^{i}z_3^{j}z_4^{k}$.  This  can only happen when
  there exists some relation $f=z_1^{r}-vz_2^{i}z_3^{j}z_4^{k}$ in $R(I)$ with a monomial $v\in K[x,y]$. Then $u_{1}^{r}=vu_2^iu_3^ju_4^k$, that is,
  $x^{cr}=vx^{bi+aj}y^{ai+bj+ck}$. This implies  that $x$ divides  $y$, a contradiction.

By symmetry,   type (3) reduces to   $z_1^{i}z_2^{j}$,   $z_1^{i}z_2^{j}-z_3^{k}z_4^{\ell}$ with $i,\ell>0$.  This  can only happen if there exists a
relation $f=z_1^{i}z_2^{j}-vz_3^{k}z_4^{\ell}$ in $R(I)$ with a monomial $v\in K[x,y]$. Then $u_1^{i}u_2^{j}=vu_3^{k}u_4^{\ell}$, that is,
$x^{ci+bj}y^{aj}=vx^{ak}y^{bk+c\ell}$. This implies  that $ci+bj\geq ak$, $aj\geq bk+c\ell$, $i+j=k+\ell$,
and the equalities hold if and only if $v=1$. It follows  that $i-\ell\geq \frac{a+b}{c}(k-j)>k-j$, a contradiction.

Also by symmetry,   type (4) reduces to $z_1^{i}z_2^{j}z_4^{k}$,  $z_3^{r}-z_1^{i}z_2^{j}z_4^{k}$  with $i>0$.  This  can only happen if there exists some relation
$f=z_1^{i}z_2^{j}z_4^{k}-vz_3^{r}$ in $R(I)$ with a monomial $v\in K[x,y]$. This implies that $ci+bj\geq ar$, $aj+ck\geq br$, $i+j+k=r$,
and the above equalities hold if and only if $v=1$. This implies that $i+k\geq \frac{a+b}{c}(r-j)>r-j$, a contradiction.

 For type (5),  if $z_1^{i}z_4^{j}$ is a  minimal generator of  $J$, then  there exists a relation  $f=z_1^{i}z_4^{j}-vz_2^kz_3^{\ell}$  in $R(I)$ with a
 monomial $v\neq 1$. It follows that $ci\geq bk+a\ell$, $cj\geq ak+b\ell$ and $i+j=k+\ell$. This implies that $i+j\geq \frac{a+b}{c}(k+\ell)>k+\ell$, a
 contradiction.

 If $z_1^{i}z_4^{j}-z_2^{k}z_3^{\ell} \in J$, then $ci=bk+a\ell$, $cj=ak+b\ell$ and $i+j=k+\ell$. It follows that $i+j=0$,
 a contradiction.

Case (ii) also can happen for example, when  $I=(x^{10},x^9y^2,x^2y^9, y^{10})$. In this example,  $J=(z_2z_3, z_2^5, z_3^5,z_1z_3^{4},z_2^{4}z_4)$.

This discussion shows that $J$ is indeed of the form (i) or (ii), when $J$ does not contain a binomial generator.

Finally we show that if the minimal generators of $J$ contain some  binomials, then $J$ must be of the form (iii)  or (iv). Indeed, by the above discussions, we
know that the   possible forms of monomial generators and binomial generators are $z_1^iz_3^j$, $z_2^kz_4^{\ell}$ with $i,j,k,\ell>0$ and
$z_1^{i}z_3^{j}-z_2^{k}z_4^{s}$ with $i,j>0$  respectively.

If $z_1^{i}z_3^{j}-z_2^{k}z_4^{s}$ with $i,j>0$ is a minimal generator of  $J$, then $ci+aj=bk$, $bj=ak+cs$ and $i+j=k+s$.
It follows that $i=s$ and $k=j$. In this case, we obtain that $ci=(b-a)j$, that is  $\frac{c}{\gcd(c,b-a)}i=\frac{b-a}{\gcd(c,b-a)}j$. This yields that
$j=\frac{c}{\gcd(c,b-a)}t$ and $i=\frac{b-a}{\gcd(c,b-a)}t$ for some integer $t$. Therefore, $z_1^{\ell}z_3^{m}-z_2^{m}z_4^{\ell}$ divides
$z_1^{i}z_3^{j}-z_2^{j}z_4^{i}$, where $\ell=\frac{b-a}{\gcd(c,b-a)}$ and $m=\frac{c}{\gcd(c,b-a)}$.

Cases (iii) and (iv)  may actually occur.  For example, when  $I=(x^{10},x^8y^3,x^3y^8, y^{10})$, then   $J=(z_2z_3, z_2^3, z_3^3,z_1z_3^{2}-z_2^{2}z_4)$;
 when  $I=(x^{30},x^{29}y^2,x^2y^{29}, y^{30})$, then   $J=(z_2z_3, z_2^{15}, z_3^{15},z_1^9z_3^{10}-z_2^{10}z_4^9,z_1z_3^{14},z_2^{14}z_4,z_1^3z_3^{13},z_2^{13}z_4^3,z_1^5z_3^{12},z_2^{12}z_4^5,z_1^7z_3^{11},z_2^{11}z_4^7)$.

To conclude the  proof of part (a) of  the theorem,  we show that the degree $r$ of the pure power generators  $z_2^{r}$, $z_3^{r}$  of $J$  in case (iii) must be
$m+1$.
It is clear  that $r\geq m+1$. If $r>m+1$, then $z_1^{\ell}z_3^{m+1}=z_3(z_1^{\ell}z_3^{m}-z_2^{m}z_4^{\ell})+z_2^{m-1}z_4^{\ell}(z_2z_3)\in J$,
contradicting the fact that no monomial different from  $z_2z_3, z_2^r, z_3^r$ belongs to a minimal set of generators of $J$.

(b) It is clear that in case (i),  $z_1$ and $z_4$ are nonzero divisors of $F(I)$. So in  this  case $F(I)$ is Cohen--Macaulay.

Next we show that   if $J$ is generated by more than three elements, then  $F(I)$ is not Cohen--Macaulay. This then shows that $F(I)$ is not
Cohen--Macaulay in the cases (ii), (iii)  and (iv).

Indeed, suppose that $\mu(J)=t\geq 4$, and that $F(I)$  is Cohen--Macaulay. Since $\height J=2$, the generators of $J$ are the maximal minors of a
$(t-1)\times t$-matrix $A$  whose entries are homogeneous polynomials of positive degree, see for example \cite[Theorem 1.4.17]{BH}. In  all four cases one of the generators is $z_2z_3$, which is  a maximal minor $A$. This implies that $t\leq 3$, a contradiction.

Finally we show that $\depth F(I)=1$ in the cases (ii),  (iii)  and (iv). It is enough to  show that $\depth F(I)>0$ in these three cases. In case (ii),  we have $J=J_1\sect J_2$,  where $J_1=(z_2,z_3 ^r,z_1^iz_3^j,\ldots)$ and
$J_2=(z_3, z_2^r,z_2^jz_4^i,\ldots)$. For $J_1$, the element $z_4$ does not appear in the support of the generators of $J_1$, and for $J_2$ it is $z_1$.
Thus $\depth T/J_1>0$ and $\depth T/J_2>0$.  Since there is a natural injective map   $F(I)=T/J\to T/J_1\dirsum T/J_2$, we see that $\depth F(I)>0$.

In case (iii),  by considering $S$-pairs,  one immediately sees that the elements
$z_2z_3, z_2^{m+1}, z_3^{m+1}, z_1^\ell z_3^m-z_2^mz_4^\ell$  form  a Gr\"obner basis with respect to the lexicographic order induced by $z_1>z_2>z_3>z_4$.
Hence $\ini_<(J)=(z_2z_3, z_2^{m+1}, z_3^{m+1}, z_1^\ell z_3^m)$. Since $z_4$ does not appear in the support of any monomial of  $\ini_<(J)$, we see that
$z_4$ is a nonzero-divisor on $T/\ini_<(J)$, and hence $\depth T/\ini_<(J)>0$. Quite generally one always has   $\depth T/\ini_<(J)\leq \depth T/J$. Thus
$\depth T/J>0$.

 Similar to case (iii), in case (iv), one  sees that the elements
$z_2z_3, z_2^{r}, z_3^{r}, z_1^\ell z_3^m-z_2^mz_4^\ell, z_1^{i'}z_3^{j'}, z_2^{j'}z_4^{i'},\ldots$  form  a Gr\"obner basis with respect to the lexicographic order induced by $z_1>z_2>z_3>z_4$.
Hence $\ini_<(J)=(z_2z_3, z_2^{r}, z_3^{r}, z_1^\ell z_3^m, z_1^{i'}z_3^{j'}, z_2^{j'}z_4^{i'},\ldots)$. We obtain that $\ini_<(J)=J'\sect J''$, where $J'=(z_2,z_3 ^r,z_1^\ell z_3^m,z_1^{i'}z_3^{j'},\ldots)$ and
$J''=(z_3, z_2^r, z_2^{j'}z_4^{i'},\ldots)$. For $J'$, the element $z_4$ does not appear in the support of the generators of $J'$, and for $J''$ it is $z_1$.
Thus $\depth T/J'>0$ and $\depth T/J''>0$.  Since there is a natural injective map   $T/\ini_<(J)\to T/J'\dirsum T/J''$, we see that $\depth T/\ini_<(J)>0$.
 Thus
$\depth T/J>0$ because of one always has   $\depth T/\ini_<(J)\leq \depth T/J$.
 \end{proof}

Next, we consider the case $b+a<c$. In this case, we have $\frac{1}{a}>\frac{1}{c-b}$. Thus there exist positive integers $i$ and $j$
such that
\begin{eqnarray}
\label{Iranfood}\frac{b-a}{c-b}j\leq i\leq \frac{b-a}{a}j.
\end{eqnarray}
Let $\Sc=\{(i,j)\in \mathbb{N}^{2}\mid \frac{b-a}{c-b}j\leq i\leq \frac{b-a}{a}j\}$ and  $\leq$ be the  partial ordering on $\Sc$ with $(i,j)\leq
(i',j')$ if and only if $i\leq i'$, $j\leq j'$.

\begin{Theorem}\label{symmetricidealdown}
Let $I=(x^c,x^by^a,x^ay^b,y^c)\subset K[x,y]$ be a symmetric monomial ideal with $b+a<c$. Let $(i,j)$ be the minimal element of the poset $\Sc$.   Then we
have:
\begin{enumerate}
\item[ (a)]  The ideal $J$ is minimally generated as given in   one of the following four cases, and each of these cases occur.
\begin{enumerate}
\item[(i)] $J=(z_1z_4, z_1^{i}z_3^{j}, z_2^{j}z_4^{i})$;
\item[(ii)] $J=(z_1z_4, z_1^{i}z_3^{j}, z_2^{j}z_4^{i})+L$, where $L$ is minimally generated by monomials of the forms $z_1^{i}z_3^{j}$  and
    $z_2^{j}z_4^{i}$  with $i,j>0$;
\item[(iii)] $J=(z_1z_4, z_1^{i}z_3^{j}, z_2^{j}z_4^{i}, z_1^\ell z_3^m-z_2^mz_4^\ell)$, where  $l=\frac{b-a}{\gcd(c,b-a)}$ and
    $m=\frac{c}{\gcd(c,b-a)}$;
 \item[(iv)] $J=(z_1z_4, z_1^{i}z_3^{j}, z_2^{j}z_4^{i}, z_1^\ell z_3^m-z_2^mz_4^\ell)+L'$, where  where $L'$ is minimally generated by monomials of the forms $z_1^{i'}z_3^{j'}$  and
    $z_2^{j'}z_4^{i'}$  with $i',j'>0$, $l=\frac{b-a}{\gcd(c,b-a)}$ and
    $m=\frac{c}{\gcd(c,b-a)}$.
    \end{enumerate}
\item[(b)] $F(I)$ is Cohen--Macaulay  if and only if $J$ is generated as in case {\em (i)}. Otherwise, $\depth F(I)=1$.
\end{enumerate}
\end{Theorem}
\begin{proof}
We omit the proof of (a) which follows the same line of arguments as the proof of part (a) of Theorem~\ref{symmetricidealup}.

(b) Recall $F(I)\cong T/J$, if $T/J$ is Cohen-Macaulay, then $\mu(J)\leq 3$, because the Hilbert-Burch matrix of $J$ must be a $2\times 3$-matrix since $z_1z_4\in J$. On the other hand, if $J=(z_1z_4, z_1^{i}z_3^{j}, z_2^{j}z_4^{i})$, then $J$ is the ideal of $2$-minors of the matrix
\[
  \left( {\begin{array}{ccc}
   z_{1}^{i-1}z_3^j, & z_4,& 0 \\
   z_{2}^iz_4^{j-1}, & 0, &z_1 \\
\end{array} } \right).
\]
This implies that   $F(I)$ is Cohen--Macaulay.
\end{proof}

\begin{Corollary}\label{symmetricidealon}
Let $I=(x^c,x^by^a,x^ay^b,y^c)\subset K[x,y]$ be a symmetric monomial ideal.  Then
$\depth F(I)>0$.
\end{Corollary}
\begin{proof}
If $a+b\neq c$, then the assertion is shown  in Theorem~\ref{symmetricidealup} and  Theorem~\ref{symmetricidealdown}. If $a+b=c$,  then
$F(I)=K[x^c,x^by^a,x^ay^b,y^c]$ is a domain, so that $\depth F(I)>0$, also in this case.
\end{proof}

\begin{Corollary}
\label{noteasy}
Let $I=(x^c,x^by^a,x^ay^b,y^c)\subset K[x,y]$ be a symmetric monomial ideal. Then the following conditions are equivalent:
\begin{enumerate}
\item[(a)]
$F(I)$ is Cohen--Macaulay;
\item[(b)]
$\mu(J)\leq 3$;
\item[(c)]
$\mu(J)= 3$.
\end{enumerate}
\end{Corollary}
\begin{proof}
The equivalence of (a) and (b) is shown in the above theorems when $a+b\neq c$. In the case that $a+b=c$, it is shown in \cite[Lemma 4]{BSV}. The equivalence of (b) and (c) also follows form the above theorems.
\end{proof}

\medskip
\section{The fiber cone of $I=(x^{2a},x^{a}y^{b},x^{c}y^{d},y^{2b})$}
We consider a special class of ideals $I\subset  K[x,y]$, whose fiber cone is a complete intersection.
In this section we assume that we are  given positive integers  $a,b,c,d$ such that  $\gcd(a,c)=1$,  $\gcd(b,d)=1$  and $b\geq a>c$.

First we consider the case $bc+ad>2ab$. Thus  we have  $\frac{2b}{2b-d}-\frac{2a}{c}>0$ and $\frac{b}{2b-d}-\frac{a}{c}>0$.
  Therefore, there exist positive integers $i$, $j$ and $s$
such that
\begin{eqnarray}
\label{Dusseldorfineq}\frac{2a}{c}i+\frac{a}{c}j\leq s\leq \frac{2b}{2b-d}i+\frac{b}{2b-d}j.
\end{eqnarray}

Let $\Tc=\{(i,j,s)\in \mathbb{N}^{3}\mid \frac{2a}{c}i+\frac{a}{c}j\leq s\leq \frac{2b}{2b-d}i+\frac{b}{2b-d}j\}$, and let  $r$ be the smallest integer such that  $(i,j,r)\in \Tc$. Then we have
\begin{Theorem}\label{Romania}
Let $I=(x^{2a},x^{a}y^{b},x^{c}y^{d},y^{2b})\subset K[x,y]$ be the   monomial ideal with $bc+ad>2ab$.  Let $K[z_1,z_2,z_3,z_4]$ be the polynomial ring, and let $J$ be the kernel of the canonical map $T\to F(I)$ with  $z_i\mapsto u_i$ for $i=1, \ldots,4$,   where    $u_1=x^{2a}$, $u_2=x^ay^b$, $u_3=x^cy^d$ and $u_4=y^{2b}$. Then $J=(z_2^2-z_1z_4, z_3^{r})$. In particular, $J$ is a complete intersection.
\end{Theorem}
\begin{proof} Notice that $u_2^2=u_1u_4$, which implies that $z_2^2-z_1z_4\in J$. By the choice of $r$, we have
 $cr\geq 2ai+aj$ and $dr\geq bj+2bk$, where $k=r-i-j$. Moreover one of the two inequalities must be strict, because in (\ref{Dusseldorfineq}) one of the two
inequalities must be strict. This implies that $z_3^{r}-vz_1^{i}z_2^{j}z_4^{k}$ is a relation of $R(I)$ with a monomial $v\neq 1$. Hence $z_3^r \in J$.

It remains to be shown that $J\subseteq (z_2^2-z_1z_4, z_3^{r})$. We know that the generators of $J$ are monomials or binomials arising from relations $f=u\prod_{i=1}^{4}z_i^{r_i}-v\prod_{i=1}^{4}z_i^{s_i}$ of the Rees ring by reduction modulo $(x,y)$. Thus  the possible minimal generators of $J$ must  be of the form  as follows:
 (1)  $z_i^r$, with $r\geq 2$ for $i=1,\ldots,4$; (2)   $z_1^{r}-z_2^iz_3^jz_4^{r-i-j}$,
$z_2^{r}-z_1^iz_3^jz_4^{r-i-j}$, $z_3^{r}-z_1^iz_2^jz_4^{r-i-j}$, $z_4^{r}-z_1^iz_2^jz_3^{r-i-j}$  with $r\geq 2$; (3) $z_1^iz_2^jz_3^{k}$, $z_1^iz_2^jz_4^{k}$,   $z_1^iz_3^jz_4^{k}$,  $z_2^iz_3^jz_4^{k}$ with $i,j,k>0$; (4)  $z_1^iz_2^j-z_3^{k}z_4^{\ell}$,  $z_1^iz_3^j-z_2^{k}z_4^{\ell}$,
 $z_1^iz_4^j-z_2^{k}z_3^{\ell}$ with $i+j>0$; (5) $z_1^iz_2^j$, $z_3^{i}z_4^{j}$,  $z_1^iz_3^j$,  $z_2^{i}z_4^{j}$,
 $z_1^iz_4^j$, $z_2^{i}z_3^{j}$ with $i,j>0$.

 Some of these monomials and binomials listed above do belong to $J$, but are not minimal  except $z_2^2-z_1z_4, z_3^{r}$ by direct calculations.
We demonstrate this in some cases. In the  remaining cases the arguments are similar.  It is obvious that $z_1^r, z_4^r,z_1^{r}-z_2^iz_3^jz_4^{r-i-j}, z_4^{r}-z_1^iz_2^jz_3^{r-i-j}$ do not belong to $J$.

If   $z_2^r\in J$, then there exists some relation $f=z_2^r-vz_1^iz_3^jz_4^{k}$ in $R(I)$ with $r=i+j+k$ and a monomial $v\neq 1$.
It follows that $ar\geq 2ai+cj$, $br\geq dj+2bk$ and one of the two inequalities must be strict. Thus we have  $(\frac{c}{2a}+\frac{d}{2b}-1)j<0$, this implies  that $j<0$, a contradiction.

If $z_3^{r}-z_1^iz_2^jz_4^{k}\in J$, where  $k=r-i-j$,  then $cr=2ai+aj$ and $dr=bj+2bk$. It follows that  $(\frac{c}{2a}+\frac{d}{2b}-1)r=0$, this implies  that $r=0$, a contradiction.

If   $z_1^iz_2^jz_3^{k}\in J$   with $i,j,k>0$, then there exists some relation $f=z_1^iz_2^jz_3^{k}-vz_4^{r}$ in $R(I)$ with $r=i+j+k$ and a monomial $v\in K[x,y]$.
It follows that $x^{2ai+aj+ck}\mid v$,  $bj+dk\geq 2br$. Thus we have  $2bi+bj+(2b-d)k\leq 0$, this implies  that $i=j=k=0$ or at least one of $i,j,k$ is negative, a contradiction.

If  $z_1^iz_3^jz_4^{k}$  with $i,j,k>0$  is  a minimal generator of $J$,  then there exists some relation $f=z_1^iz_3^jz_4^{k}-vz_2^{r}$ in $R(I)$ with $r=i+j+k$ and a monomial $v\in K[x,y]$. It follows that $2ai+cj\geq ar$,  $dj+2bk\geq br$. We consider three cases:

(i) $i=k$, then we have $2ai+cj\geq a(2i+j)$. This implies that $cj\geq aj$. From the fact $c<a$, we can get $j=0$,  a contradiction.

(ii) $i<k$, then  $z_2^{2i}z_3^jz_4^{k-i}=z_1^iz_3^jz_4^{k}-((z_1z_4)^{i}-z_2^{2i})z_3^{j}z_4^{k-i}$  is  a minimal generator of $J$,  a contradiction.

(iii) $i>k$, this case can be proved as similar to (ii).

If   $z_1^iz_2^j-z_3^{k}z_4^{\ell}\in J$  with $i+j>0$, then $2ai+aj=ck$, $bj=dk+2b\ell$ and $i+j=k+\ell$. This implies that $i=j=k=\ell=0$, a contradiction.

If $z_1^iz_3^j$ with $i,j>0$ is  a minimal generator of $J$. Then  $j<r$, and  there exists some relation $f=z_1^iz_3^j-vz_2^kz_4^{\ell}$ in $R(I)$ with $i+j=k+\ell$ and a monomial $v\neq 1$. It follows that  $u_1^iu_3^j=vu_2^ku_4^{\ell}$. We consider three cases;
 (i)  $k=0$, then we have $dj\geq 2b\ell$. Hence $j>\ell$ because $d<2b$, a contradiction.
 (ii) $k>0$ and $i\geq k$, then  $u_1^{i-k}u_3^{j}=vu_4^{k+\ell}$. This implies that $dj\geq2b(k+\ell)$. From the fact that $d<2b$, we have $j>k+\ell$, a contradiction.
 (iii)  $k>0$ and $i<k$, then $u_3^{j}=vu_1^{k-i}u_4^{k+\ell}$. By the choice of $r$, we know that $j\geq r$, a contradiction.

Finally, if $z_2^{r}-z_1^iz_3^jz_4^{k}\in J$, then we have $ar=2ai+cj$, $br=dj+2bk$ and $r=i+j+k$. It follows that $i=k$ and $j=0$. Hence
 $z_2^{2}-z_1z_4$ divides $z_2^{r}-z_1^iz_3^jz_4^{k}$.
\end{proof}

Next, we consider the case $bc+ad<2ab$. In this case, we have $0<(2-\frac{d}{b})-\frac{c}{a}<1$. Thus there exist positive integers $s$ and $\ell$
such that
\begin{eqnarray}
\label{inequality1} \frac{c}{a}s\leq \ell\leq (2-\frac{d}{b})s.
\end{eqnarray}
Let $\D=\{(s,\ell)\in \mathbb{N}^{2}\mid \frac{c}{a}s\leq \ell\leq (2-\frac{d}{b})s\}$,  and  let   $r$ be the smallest integer such that $(r,\ell)\in \D$.
We set $i=\frac{\ell}{2}$ and $j=0$ if $\ell$ is even, and  $i=\frac{\ell-1}{2}$ and $j=1$, if  $\ell$ is odd.

With the same methods as in the proof of Theorem~\ref{Romania} one obtains

\begin{Theorem}\label{kolon}
Let $I=(x^{2a},x^{a}y^{b},x^{c}y^{d},y^{2b})\subset K[x,y]$ be the   monomial ideal with $bc+ad<2ab$.
We write  $F(I)\cong K[z_1,z_2,z_3,z_4]/J$ as in Theorem~\ref{Romania}. Then  $J=(z_2^2-z_1z_4, z_1^{i}z_2^{j}z_4^{r-i-j})$. In particular, $J$ is a complete intersection.
\end{Theorem}
\begin{proof}  Notice that $u_2^2=u_1u_4$, which implies that $z_2^2-z_1z_4\in J$. By the choice of $i,j,r$, we have $\ell=2i+j$. Set $k=r-i-j$. We obtain that $2ai+j\geq cr$, $bj+2bk\geq dr$.  Moreover one of the two inequalities must be strict, because in (\ref{inequality1}) one of the two inequalities must be strict. This implies that $z_1^{i}z_2^{j}z_4^{r-i-j}-vz_3^{r}$ is a relation of $R(I)$ with a monomial $v\neq 1$. Hence $z_1^{i}z_2^{j}z_4^{r-i-j}\in J$.

The proof for the fact that  $J\subseteq (z_2^2-z_1z_4, z_1^{i}z_2^{j}z_4^{r-i-j})$ is a case by case discussion as in Theorem~\ref{Romania} which we omit.
\end{proof}

\begin{Theorem}\label{Hibinotcome}
Let $I=(x^{2a},x^{a}y^{b},x^{c}y^{d},y^{2b})\subset K[x,y]$ be  the monomial ideal with $bc+ad=2ab$.
Then  $F(I)\cong K[z_1,z_2,z_3,z_4]/J$,  where $J=(z_2^2-z_1z_4, z_3^{a}-z_1^{i}z_2^{j}z_4^{a-i-j})$.  Here
$i=\frac{c}{2}$ and $j=0$,  if $c$ is even, and $i=\frac{c-1}{2}$ and $j=1$,  if $c$ is odd. In particular, $J$ is a complete intersection.
\end{Theorem}
\begin{proof}  Notice that $u_2^2=u_1u_4$, which implies that $z_2^2-z_1z_4\in J$. Next we show that $z_3^{a}-z_1^{i}z_2^{j}z_4^{a-i-j}\in J$.
Indeed, if  $c$ is even, then $c=2c'$. It follows that $ad=2ab-bc=2b(a-c')$. Hence
$u_3^a=x^{ac}y^{ad}=x^{2ac'}y^{2b(a-c')}=u_1^{c'}u_4^{a-c'}=u_1^{\frac{c}{2}}u_4^{a-\frac{c}{2}}=u_1^{i}u_2^{j}u_4^{a-i-j}$ where $i=\frac{c}{2}$ and $j=0$.
If  $c$ is odd, then $c=2c'+1$, i.e., $c'=\frac{c-1}{2}$. It follows that $ad=2ab-bc=2b(a-c')-b$. Therefore,
$u_3^a=x^{ac}y^{ad}=x^{ac}y^{2b(a-c')-b}=x^{a(2c'+1)}y^{2b(a-c')-b}=x^{a(2c'+1)}y^{2b(a-c'-1)+b}
=u_1^{\frac{c-1}{2}}u_2u_4^{a-\frac{c+1}{2}}=u_1^{i}u_2^{j}u_4^{a-i-j}$, where $i=\frac{c-1}{2}$ and $j=1$.

Let $L=( z_2^2-z_1z_4, z_3^{a}-z_1^{i}z_2^{j}z_4^{a-i-j})$. Since $\ini_<(L)=(z_2^2,z_3^a)$ with respect to the  lexicographic order induced by $z_3>z_2>z_1>z_4$, we see that  $L$  is generated by a regular sequence, and hence $\height(L)=2$. We will  show that $L$ is a prime ideal. This will then prove that $L=J$,  since $L\subseteq J$ and $J$ is also a prime ideal of height $2$, because $F(I)\iso K[x^{2a},x^{a}y^{b},x^{c}y^{d},y^{2b}].$

In order to see that $L$ is a prime ideal, we first observe that $z_1$ is a nonzero-divisor modulo $L$, because it is a nonzero-divisor modulo $\ini_<(L)$. Therefore,  by localization,  we obtain an injective map $T/L\to (T/L)_{z_1}$.  Thus it suffices to show that $(T/L)_{z_1}$ is a domain.  Note that $L_{z_1}=(z_2^2z_1^{-1}-z_4, f)$ where $f=z_3^{a}-z_1^{i}z_2^{j}z_4^{a-i-j}$.
Replacing $z_4$ by $z_2^2z_1^{-1}$ in $f$, we obtain that $T_{z_1}/L_{z_1}=K[z_1^{\pm 1},z_2,z_3]/(g)$
where
$(g)=(z_3^a-z_1^{i}z_2^{j}(z_2^2z_1^{-1})^{a-i-j})=(z_3^a-z_1^{-a+2i+j}z_2^{2a-2i-j})=(z_2^{2a-2i-j}-z_3^az_1^{a-2i-j})=(z_2^{a-c}-z_1^{a-c}z_3^{a})$,
where $c=2i+j$.

 The $K$-algebra $A=K[z_1,z_2,z_3]/(h)$ is a domain, where  $h=z_2^{a-c}-z_1^{a-c}z_3^{a}$,  because $\gcd(a,c)=1$. Indeed, let $v=(c-a,a-c,-a)\in  \ZZ^3$, and let $L=\ZZ v$ the sublattice of  $\ZZ^3$ spanned by $v$. Then $(h)$ may be viewed as the lattice ideal of $L$. The condition $\gcd(a,c)=1$ implies that $\ZZ^3/L$ is torsionfree. Therefore, by \cite[Theorem 3.17]{HHO} it follows that $A$ is a toric ring, and hence a domain. Now this implies that  $T_{z_1}/L_{z_1}=K[z_1^{\pm 1},z_2,z_3]/(g)$ is a domain, because it is isomorphic to $A_{z_1}$.
\end{proof}

\medskip

{\bf Acknowledgement.}  This paper is supported by the National Natural Science Foundation of
China (No. 11271275) and by the Foundation of the Priority Academic Program Development of Jiangsu Higher Education Institutions.

\end{document}